\documentclass[a4paper]{amsart}
\usepackage{amscd,amsmath,amssymb,amsthm}
\usepackage[all]{xy}
\usepackage{tikz,tikz-cats}

\newcommand{\RR}{{\mathbb R}}
\newcommand{\ZZ}{{\mathbb Z}}
\newcommand{\III}{
I }
\newcommand{\UI}{
{\mathbb I}}
\newcommand{\SDTop}{
\mathit{SDTop} }
\newcommand{\DTop}{
\mathit{DTop} }
\newcommand{\Top}{
\mathit{Top} }
\newcommand{\STop}{
\mathit{STop} }
\newcommand{\LL}{
L }
\newcommand{\saturate}{
\mathit{sat} }

\newcommand{\id}{\mathit{id}}
\newcommand{\ev}{\mathit{ev}}
\newcommand{\cal}[1]{\mathcal{#1}}
\newtheorem{theorem}{Theorem}[section]

\newtheorem{prop}[theorem]{Proposition}
\newtheorem{defi}[theorem]{Definition}
\newtheorem{remark}[theorem]{Remark}
\newtheorem{example}[theorem]{Example}
\newtheorem{examples}[theorem]{Examples}

\begin{document}

\title{Saturating directed spaces}
 
\author{Andr\'e Hirschowitz} \curraddr{Universit\'e de Nice - Sophia Antipolis}
 
\author{Michel Hirschowitz} \curraddr{CEA - LIST}
 
\author{Tom Hirschowitz} \curraddr{CNRS, Universit\'e de Savoie}
\dedicatory{}
\thanks{This work owes much to the second author's wedding on October 30, 2010.}
 
\maketitle
\section{Introduction}
 
Directed algebraic topology~\cite{Grandis} (see also, e.g.,
\cite{Pratt:1991:MCG:99583.99625,DBLP:conf/concur/GoubaultJ92,Grandis,Gaucher03,Fajstrup2006241,DBLP:journals/acs/Krishnan09,Worytkie:sheaves})
has recently emerged as a variant of algebraic topology. In the
approach proposed by Grandis, a \emph {directed} topological space (or
\emph{d-space} for short), is a topological space equipped with a set
of \emph{directed} paths satisfying three conditions. These conditions
are the three conditions necessary for constructing the so-called
fundamental category: constant paths are directed for having
identities, stability under concatenation is required for having
composition, and reparameterisation is required for having
associativity.  These are somehow minimal conditions, which leave
room for a lot of exotic examples (see Examples~\ref{ex:exotic}).
 
In the present work we propose an additional condition of saturation
for distinguished sets of paths and show how it allows to rule out
exotic examples without any serious collateral damage.
 
Our condition involves ``directed'' functions (to the unit interval
$\III$), namely those which are non-decreasing along each directed
path. And it asserts that a path along which any such (local) directed
function is non-decreasing should be directed itself.
 
Our saturation condition is local in a natural sense, and is satisfied
by the directed interval (and the directed circle). Furthermore we
show in which sense it is the strongest condition fulfilling these two
basic requirements.
 
Our saturation condition selects a full subcategory $\SDTop$ of the
category $\DTop$ of d-spaces, and we show that this new category has
all standard desirable properties, namely:
 
\begin{itemize}
\item
  $\SDTop$ is a full, reflective subcategory of $\DTop$, wich means that
  there is a nice saturation functor from $\DTop$ to $\SDTop$;
 
\item it is closed under arbitrary limits;
 
\item although it is not closed under colimits (as a subcategory), it has arbitrary
  colimits, each of which is obtained by saturation of the
  corresponding colimit in $\DTop$;
 
\item $\SDTop$ is a $\mathit{dIP1}$- category in the sense
  of~\cite{Grandis} which essentially means that it has nice cylinder
  and cocylinder constructions;
 
\item the forgetful functor from $\SDTop$ to $\Top$ has both a right and
  a left adjoint.
 
\end{itemize}

  Altogether these properties satisfy the {\em general principles}
  which, according to Grandis~\cite[Section 1.9]{Grandis}, should be
  satisfied by a {\em good topological setting} for directed algebraic
  topology.

  In Section~\ref{sec:cond} we describe our sheaf of ``directed''
  functions, and introduce our saturation condition.  In
  Section~\ref{sec:adj}, we exhibit adjunctions relating our new
  category $\SDTop$ to $\Top$ and $\DTop$.  In Section~\ref{sec:lims},
  we prove the completeness and cocompleteness properties of $\SDTop$.
  In Section~\ref{sec:cyl}, we prove that $\SDTop$ admits cylinder and
  cocylinder constructions with the desired properties.  Finally in
  Section~\ref{sec:other}, we discuss other saturation conditions, and
  show in which sense ours is maximal.

\section{Saturated d-spaces}\label{sec:cond}
 
We denote by $I$ the standard closed unit interval, and by $\DTop$ the category where
\begin{itemize}
\item objects  
 are all \emph{directed} spaces, i.e., pairs $(X, dX)$ of
  a topological space $X$ and a set $dX$ of continuous maps $I \to X$,
subject to the following three conditions
  \begin{itemize}
  \item constant paths are in $dX$,
  \item $dX$ is stable under concatenation,
\item $dX$ is stable under precomposition with continuous, non-decreasing maps
  $I \to I$;
  \end{itemize}
 
\item morphisms from $(X, dX)$ to $(Y, dY)$ are all continuous maps $f \colon
  X \to Y$ satisfying $f \circ dX \subseteq dY$.
\end{itemize}
The set $dX$ is called the set of \emph{directed} paths, or \emph{d-paths} in $(X, dX)$.  
 
In the sequel, for a $d$-space $X$, we will also write $X$ for the underlying topological space, and we will write  $dX$
for its set of directed paths.
 
We denote by $\UI$ the $d$-space obtained by equipping the standard closed unit interval
with the set of non-decreasing (continuous) paths.
 
\begin{examples}\label{ex:exotic}
As promised, here are a few exotic examples. For each of them, the underlying space is either the usual plane $X :=\RR^2$ or its quotient, the standard torus
$\overline X := \RR^2 / \ZZ^2$, which we consider equipped with the usual product order (or local order). Hence we just specify  
the distinguished subset of paths, either $dX$ or $d\overline X$.

\begin{enumerate}
\item 
$dX$ consists of all horizontal
  paths with rational ordinate (i.e., continuous maps $p \colon \III
  \to \RR^2$ with $p(t) = (q(t),a)$, for some rational $a$ and
  continuous $q \colon \III \to \RR$). 
\item 
 $d\overline X$ consists of all horizontal  
 paths with rational ordinate (i.e., continuous maps $p \colon \III
  \to \RR^2 / \ZZ^2$ with $p(t) = (q(t),\overline a)$, for some rational $a$ and
  continuous $q \colon \III \to \RR / \ZZ$). 
\item 
 $dX$ consists of all horizontal nondecreasing
  paths with rational ordinate (i.e., continuous maps $p \colon \III
  \to \RR^2$ with $p(t) = (q(t),a)$, for some rational $a$ and
  continuous nondecreasing $q \colon \III \to \RR$). 
\item 
$d\overline X$ consists of all horizontal locally nondecreasing
  paths with rational ordinate (i.e., continuous maps $p \colon \III
  \to \RR^2 / \ZZ^2$ with $p(t) = (q(t),\overline a)$, for some rational $a$ and
  continuous locally nondecreasing $q \colon \III \to \RR / \ZZ$). 
\item  
  $dX$
  consists of piecewise horizontal or vertical paths, i.e., (finite)
  concatenations of vertical and horizontal paths.
\item  
  $d \overline X$
  consists of piecewise horizontal or vertical paths.
\item  
  $dX$
  consists of continuous $p \colon \III \to
  \RR^2$ whose restriction to some dense open $U \subseteq \RR^2$ is
  locally piecewise horizontal or vertical.

\item 
 $dX$
  consists of piecewise rectilinear paths
  with rational slope. (More generally, for any subset $P$ of $\RR$
  containing at least two distinct elements, piecewise rectilinear
  paths in $\RR^2$ with slope in $P$ form a d-space.)
\item  
  $dX$
  consists of piecewise circular paths.
 
\item  
  $dX$
  consists of piecewise horizontal or vertical nondecreasing paths.

\item  
  $d \overline X$
  consists of piecewise horizontal or vertical locally nondecreasing paths.

\end{enumerate}
\end{examples}
 
We now describe our saturation process, which will rule out such
examples.
 
For any $d$-space $X$, we have the sheaf $\hat I_X$ on $X$ which assigns to
any open $U \subseteq X$ the set of
continuous functions $U \to I$. Note that each such $U$ inherits a structure of $d$-space. We refer to this structure by saying that $U$ is
an open subspace of $X$.

\begin{defi}
  For any d-space $X$, we denote by $\hat \UI_X$
 the subsheaf of $\hat I_X$ consisting, on any open subspace
  $U \subseteq X$, of all morphisms of $d$-spaces $c \colon U \to \UI$.
\end{defi}
 
We say that such a section  $c \colon U \to \UI$ of this sheaf is a \emph{directed} function (on $U$).
 
The statement that this is indeed a subsheaf needs a proof:

\begin{proof}
 Let us consider an open subspace $U \subseteq X$, a continuous map $c \colon U \to \III$,
  and an open covering $(U_j)_{j \in J}$ of $U$ such that any restriction
  $c_j$ of $c$ to a $U_j$ is a directed function. In order to prove that $c$ is directed,  we consider
  an arbitrary directed path $p \colon \III \to U$ in $dX$ and show that $c \circ p$ is
  non-decreasing.  Pulling back the covering along $p$ gives a covering
  $(V_j)_{j \in J}$ of $\III$, and the restrictions $p_j \colon V_j
  \to U_j$ of $p$ are locally non-decreasing. We conclude by recalling that a function which is locally non-decreasing on $\III$ is globally non-decreasing.\end{proof}
 
\begin{examples}\ 
\begin{itemize}
\item
 On the directed interval $\UI$, the sheaf of directed functions is the sheaf of locally non-decreasing functions.  
 
\item The directed circle is a locally ordered space, and its sheaf of directed functions is the sheaf of locally non-decreasing functions.  

\item On Examples~\ref{ex:exotic}, for  items 1 and 2,
the sheaf of
  directed functions is the sheaf of all
  continuous functions which are locally horizontally constant.

\item  On Examples~\ref{ex:exotic}, for  items 3 and 4,
the sheaf of
  directed functions is the sheaf of
functions which are locally nondecreasing in the first variable.

\item  On Examples~\ref{ex:exotic}, for  items 5 to 9,
the sheaf of
  directed functions is the sheaf of locally constant
functions.

\item  On Examples~\ref{ex:exotic}, for  items 10 and 11,
the sheaf of
  directed functions is the sheaf of locally nondecreasing
functions.
\end{itemize}

\end{examples}
 
Morphisms of $d$-spaces respect directed functions in the following sense:

\begin{prop}\label{prop:pbkdir}
 
  Let $f \colon X \to Y$ be a morphism of $d$-spaces. If $X' \subseteq X$
  and $Y' \subseteq Y$ are open subspaces with $f(X') \subseteq Y'$,
  then for any directed function $c \colon Y' \to \III$ on $Y'$, $c \circ f$ is a
  directed function on $X'$.
 
\end{prop}
 
\begin{proof}
The point is that $f$ induces a morphism from $X'$ to $Y'$. Since $p$ is a morphism from $Y'$ to $\UI$, the composite $ p \circ f$ is a morphism from $X'$ to $\UI$.\end{proof}
\begin{remark}The previous statement has a sheaf-theoretic formulation
  as follows: the continuous $f \colon  X \to Y$ yields a companion sheaf
  morphism $f^* \colon  \hat I _Y \to f_* \hat I _X$ and if $f$ is a
  morphism, then $f^*$ sends the sheaf of directed functions on $Y$
  into the (direct image of the) sheaf of directed functions on $X$.
 
\end{remark}
 
Next we introduce our notion of weakly directed paths:
\begin{defi}
  We say that a path $c \colon  I \to X$ in a $d$-space $X$ is \emph{weakly
    directed} if, given any directed function $f\colon  U \to \III$ on an
  open subspace $U \subseteq X$, $f \circ c \colon  c^{-1}(U) \to \III$ is
  again directed, that is to say locally non-decreasing.
 
We denote by $\hat d X$ the set of weakly directed paths in $X$.
\end{defi}
\begin{remark}
 
Note that the inverse image $c^{-1}(U)$ need not be connected, so that the pull-back $f \circ c$ may be locally non-decreasing without being globally non-decreasing.
 
\end{remark}
 
\begin{example}
  Of course directed paths are also weakly directed. Here we sketch an
  example of a weakly directed path which is not directed. Consider
  the plane $\RR ^2$, equipped with the set of piecewise horizontal or
  vertical paths. Its directed functions are locally constant
  functions. As a consequence, all its paths are weakly directed.
 
\end{example}
 
We are now ready for the introduction of our saturation condition.
\begin{defi}
  We say that a $d$-space $X$ is \emph{saturated} if each weakly
  directed path in $X$ is directed, in other words if $\hat d X = dX$.
\end{defi}
 
\begin{examples}\label{ex:sat}
\ 
\begin{itemize}
\item
 On the directed interval $\UI$, the sheaf of directed functions is the sheaf of locally non-decreasing functions and $\UI$ is saturated.
 
\item Since its sheaf of directed functions is the sheaf of locally non-decreasing functions, the directed circle is saturated.  For this example, the consideration of the sheaf instead of only global directed functions is obviously crucial.

\item
 The Examples~\ref{ex:exotic} are nonsaturated. We will see below what is their \emph{saturation}.
\end{itemize}

\end{examples}

\section{Adjunctions}\label{sec:adj}
 
We now have the full subcategory $\SDTop$ consisting of saturated $d$-spaces, which is equipped with the forgetful functor $U \colon  \SDTop \to Top$. This functor
has a right adjoint which sends a space $X$ to the $d$-space obtained by equipping $X$ with the full set of paths in $X$. We will see below that $U$ also has a left adjoint, obtained as a composite
of the left adjoint to $U \colon  \DTop \to \Top$ and the left adjoint $\LL$ to $\SDTop \to \DTop$ which we build now.

\begin{defi}  Given a $d$-space $X :=(X, dX)$, we build its saturation $\hat X :=(X, \hat dX)$ (recall that $\hat d X$ is the set of weakly directed paths in $X$).
 
\end{defi}
 
\begin{examples}\ 
\begin{itemize}

\item On Examples~\ref{ex:exotic}, for  items 1 and 2,
weakly
  directed paths are horizontal paths.

\item  On Examples~\ref{ex:exotic}, for  items 3 and 4,
weakly
  directed paths are (locally) nondecreasing horizontal paths.

\item  On Examples~\ref{ex:exotic}, for  items 5 to 9,
all  paths are weakly directed.

\item  On Examples~\ref{ex:exotic}, for  items 10 and 11,
weakly
  directed paths are all (locally) nondecreasing  paths.
\end{itemize}

\end{examples}

In the previous definition, the verification that $\hat X$ is indeed a
saturated $d$-space is straightforward. It is also easy to check that
this construction is functorial, that is to say that if the continuous
map $c\colon X \to Y$ is a morphism of $d$-spaces, then it is a
morphism from $\hat X$ to $\hat Y$ as well. This yields our left
adjoint $\LL\colon \DTop \to \SDTop$ for the inclusion $J\colon \SDTop
\to \DTop$. Indeed $\LL \circ J$ is the identity, and we take the
identity as counit of our adjuntion. While for the unit $\eta$
evaluated at $X$, we take the identity map: $id_X\colon X \to \hat
X$. The equations for these data to yield an adjunction (see
\cite[Chapter IV, Thm 2 (v)]{MacLane:cwm}) are easily verified. Since
$J$ is an inclusion, this adjunction is a so-called \emph{reflection}
(a \emph{full} one since $\SDTop$ is by definition full in $\Top$).
 
\begin{remark}
  We could build a more symmetric picture as follows. There is a
  category $\STop$ consisting of spaces equipped with a ``structural''
  subsheaf of their sheaf $\hat{I}_X$ of functions to $I$. Morphisms
  are those continuous maps $f \colon X \to Y$ which, by composition,
  send the structural subsheaf on $Y$ into the structural subsheaf of
  $X$.  (In sheaf-theoretical terms, maps $f$ such that $D_Y
  \hookrightarrow \hat{I}_Y \to f_* (\hat{I}_X)$ factors through $f_*
  (D_X) \hookrightarrow f_* (\hat{I}_X)$, where $D_X$ and $D_Y$ are
  the structural sheaves of $X$ and $Y$, respectively.)  Our
  construction of the sheaf of directed functions can be upgraded into
  a functor $D\colon \DTop \to \STop$. Dually, our construction of the
  saturation can be upgraded into a functor $S \colon \STop \to \DTop$
  which is right adjoint to $D$.  This adjunction factors through
  $\SDTop$, which is thus not only a reflective subcategory of $\DTop$
  but also a coreflective subcategory of $\STop$.  We choose not to
  develop this material here and select only the next two statements.
\end{remark}
 
\begin{prop}
 
  Let $X$ be a topological space, and $\cal D$ a subsheaf of $\hat I
  _X$. Then the set $d_{\cal D} X$ of paths in $X$ along which local
  sections of $\cal D$ are locally non-decreasing turns $X$ into a
  saturated $d$-space.
\end{prop}
 
\begin{proof}
  Indeed, it is easily checked that $(X, d_{\cal D} X)$ is a
  $d$-space, and that $\cal D$ is contained in the sheaf of
  directed functions on this $d$-space. Hence weakly directed paths in
  $(X, d_{\cal D} X)$ are automatically in $d_{\cal D} X$.\end{proof}
\begin{remark}
The previous statement allows to check easily that a $d$-space is saturated after having guessed (instead of proved) what are its directed functions.
\end{remark}
 
\begin{prop}\label{prop:dmap}
 
Let $f \colon X \to Y$ be a continuous map between saturated  $d$-spaces. If $f$ transforms, by composition, (local) directed functions on $Y$ into (local) directed functions on $X$, then $f$ is a morphism of $d$-spaces.
\end{prop}
 
\begin{proof}
Indeed, consider a directed path $c\colon  \III \to X$. we prove that $f \circ c$ is weakly directed (hence directed). for this we take a  
(local) directed function $p \colon  Y' \to \III$. We know that $p \circ f$ is a (local) directed function on $X$, hence $p \circ f \circ c$ is directed.  
\end{proof}
 
\begin{remark}
 
In the previous statement, the assumption that $X$ is saturated is useless, but we prefer to see this as a characterisation of morphisms in  $\SDTop$.
\end{remark}

\section{Completeness and cocompleteness}\label{sec:lims}
In this section, we prove that $\SDTop$ is complete and cocomplete. We
furthermore show that limits may be computed as in $\DTop$.
 
First, we have easily:
\begin{prop}
  $\SDTop$ is cocomplete.
\end{prop}
\begin{proof}
  $\DTop$ is cocomplete~\cite{Grandis}, so given any diagram $D \colon
  J \to \SDTop$, we may compute its colimit $d$ in $\DTop$. The left
  adjoint $\LL$ then preserves colimits, of course, but it also
  restores the original diagram by idempotency, so that $\LL (d)$ is a
  colimit of $D$ in $\SDTop$.\end{proof}
\begin{example}
 Here we sketch an example showing that a colimit of saturated $d$-spaces need not be saturated.
This involves four different directed planes. The first one $P_0$ has only constant directed paths. The next two ones $P'$ and $P''$ have only horizontal (resp. vertical) directed paths.
The fourth one is the coproduct $P_1$ of $P'$ and $P''$ along $P_0$. Its directed paths are piecewise horizontal or vertical. As we have seen above,   
 all its paths are weakly directed, hence it is not saturated.
\end{example}
 
\begin{example}\label{ex:satt}
 The product $\UI \times \UI$ in $\DTop$, which has as underlying
  space the product $\III \times \III$ and as directed paths all
  continuous maps $p \colon \III \to \III \times \III$ with
  non-decreasing projections, is saturated. Furthermore, for any open
  $U \subseteq \III \times \III$, a map $c \colon U \to \III$ is
  directed iff it is \emph{locally non-decreasing}, i.e., for any $x
  \in U$, there is a neighbourhood $V$ of $x$ on which $c$ is
  non-decreasing. This is of course equivalent to being locally
  \emph{separately} non-decreasing, i.e., locally non-decreasing in
  each variable.
\end{example}

Next we also have:
\begin{prop}\label{lem:SDcomplete}
  $\SDTop$ is complete as a subcategory of $\DTop$.
\end{prop}
\begin{proof}
  First, recall that $\DTop$ is complete~\cite{Grandis}.  Then,
  consider any diagram $D \colon J \to \SDTop$, and its limiting cone
  $u_j \colon d \to D_j$ in $\DTop$. Let $d' = \saturate(d)$. By
  universal property of $\eta$, this yields a cone $u'_j \colon d' \to
  D_j$ in $\SDTop$. By universal property of $d$, we also have a compatible morphism
from $d'$ to $d$, which has to be the identity.\end{proof}

\section{Towards directed homotopy}\label{sec:cyl}
As explained by Grandis~\cite{Grandis}, the basic requirement for
building directed homotopy is the existence of convenient cylinder and
cocylinder constructions. In the present section, we check that our
category $\SDTop$ is stable under the cylinder and cocylinder
constructions in $\DTop$. In Grandis's terminology, this reads as
follows:
\begin{theorem}
  $\SDTop$  is a cartesian  dIP1- category.
\end{theorem}
\begin{proof}
Since $\SDTop$ is a full cartesian subcategory of $\DTop$ which is a cartesian  dIP1- category (see~\cite[Section 1.5.1]{Grandis}), we just have to check that it is stable under the cylinder,  
the cocylinder and the reversor constructions.
 
 It is clear for the cylinder, since it is the product with the directed interval, which is an object of $\SDTop$.

For the cocylinder construction, we must check that for $X$ in  $\SDTop$, its path-object $X ^\UI$ is again in $\SDTop$. Thus we have to prove that weakly directed paths  
in  $ X ^\UI$ are directed.  

%
On the way, we have ``separately directed'' paths.  First, recall from
Grandis~\cite[Section 1.5.1]{Grandis} that for any d-space $X$ and $t
\in \III$, evaluation at $t$ yields a directed morphism $\ev_t \colon
X^\UI \to X$. Say that a path $p \colon \UI \to X^\UI$ is \emph{separately
  directed} iff for all $t \in \III$, $\ev_t \circ p$ is directed in
$X$.  
 
We first prove that any weakly directed path $p \colon I \to X ^\UI$
is separately directed. Thus we have a point $t \in I$ and we must
prove that $p_t := \ev_t \circ p$ is directed in $X$.  Since $X$ is
saturated, it is enough to show that it is weakly directed. For this,
we consider a directed function $f\colon U \to \UI$ on an open set $U
\subseteq X$, and we must prove that, where defined, $f \circ p_t$ is
locally non-decreasing. In pictures, we must prove that the top row of
 
\begin{center}
  \Diag{%
    \pbk{III}{U''}{U'} %
    \pbk{XUI}{U'}{U} %
  }{%
   |(U'')| U'' \&|(U')| U' \&|(U)| U \& |(I)| I \\
   |(III)| \III \&|(XUI)| X^{\UI} \& |(X)| X %
  }{%
    (U'') edge[labelu={i}] (U') %
    edge[into] (III) %
    (U') edge[labelu={j}] (U) %
    edge[into] (XUI) %
    (U) edge[labelu={f}] (I) %
    edge[into] (X) %
    (III) edge[labeld={p}] (XUI) %
    (XUI) edge[labeld={\ev_t}] (X) %
    (III) edge[bend right={20},labeld={p_t}] (X) %
  }
\end{center}
is locally non-decreasing. Since $p$ is weakly directed in $X^\UI$, it
is enough to show that $f \circ j$ is directed, which holds by
Proposition~\ref{prop:pbkdir}.

 
Now we prove that any separately directed path is directed.  Consider
any separately directed $p \colon \UI \to X^\UI$. It is directed in
Grandis's sense iff its uncurrying $p' \colon \UI \times \UI \to X$
is.  For this, by Example~\ref{ex:satt} ($\UI \times \UI$ is saturated)
and Proposition~\ref{prop:dmap}, it is enough to show that for any
directed map $f$ on $X$, the composite $f \circ p'$, where defined, is
directed on $\UI \times \UI$.  By Example~\ref{ex:satt} again, this is
equivalent to both $f \circ p' \circ \langle \id, t \rangle$ and $f
\circ p' \circ \langle t, \id \rangle$ being locally non-decreasing,
for any $t \in \III$.  For the first map, observe that $p' \circ
\langle \id, t \rangle = \ev_t \circ p$ is directed in $X$ by
hypothesis.  Hence, because $X$ is saturated, its composition with
$f$, where defined, is locally non-decreasing. For the second map, we
have $p' \circ \langle t, \id \rangle = p(t)$, which is directed in
$X$ by construction of $X^\UI$, hence, again its composition with $f$,
where defined, is locally non-decreasing.
 
Finally we check that $\SDTop$ is stable under reversion. For this we take a saturated $d$-space $X$ and prove that $RX$ is saturated. We first check that directed functions on $RX$ are exactly  
obtained by reversion from directed functions on $X$. Then we take a weakly directed path $c$ in $RX$ and check easily that its reversion is weakly directed on $X$ hence directed,  
which means that $c$ is directed in $RX$.
\end{proof}
\section{A universal property of saturation}\label{sec:other}
 
In this final section, we discuss other possible saturation processes, showing in which sense our choice is the best one.
 
As a first naive attempt, we could have defined weakly directed paths by testing only against global directed functions.  
In this case, the directed interval would have remained saturated; but the saturation of the directed
circle would have produced the reversible circle, which is highly undesired. This explains why we have considered local directed functions.
 
As a second, much more reasonable attempt, we can define \emph{almost directed} paths to be limits (in the compact-open topology) of directed paths.
 It is easily checked that almost directed paths are weakly directed. But we observe that almost directed paths are not in general stable by concatenation. To see this, just
equip the real line $L$ with the set $dL$ of paths which are constant or avoid $0$: almost directed paths are those which stay in the nonnegative,  
or in the nonpositive half-line.
 
Of course we could nevertheless define the small saturation of a $d$-space $X$ to be obtained by equipping $X$ with the smallest set $adX$ of paths in $X$ containing almost directed paths and  
stable by reparameterisation and concatenation. This is in general strictly smaller than the set of weakly directed paths. To show this  we sketch an ad hoc example.
 
\begin{example}
 Our example is a subspace $H$ (for harp) of $\RR^3$. It consists of a skew curve $C$, together with some of its chords $L_{a,b}$ (here, by the chord, we mean the closed segment) .  
For the curve $C$, we take the rational cubic curve:
 
$$ C := \{(t, t^2, t^3) \vert t \in \RR \}.$$
 
The interesting property of this curve is that its chords meet $C$
only at two points, and two of these chords cannot meet outside
$C$. Indeed, otherwise, the plane containing two such chords would
meet our cubic curve in four points. We pose $C_t := (t, t^2, t^3)$
and write $L_{a,b}$ for the chord through $C_a$ and $C_b$.
 
We take for $H$ the union of $C$ with the chords $L_{a,b}$ for $a<b$, $a$ rational and $b$ irrational (the point here is that these two subsets are dense and disjoint).
We take for $dH$ the set of paths which are either constant or directed paths in one of the chords $L_{a,b}$ equipped with the usual order with $C_a  < C_b$. These are clearly stable by reparameterisation  
and there is clearly no possibility for concatenation except within a chord. Thus this yields a $d$-space.
 
Concerning this $d$-space, we have two claims. We first claim that this set of paths is closed. Indeed, a (simple) limit of paths each contained in a line is contained in a line too and if the  
limiting path is not constant, the line for the limit has to be the limit of the lines. Secondly we claim that this $d$-space is not saturated. Indeed, local directed functions are non-decreasing along $C$
 (equipped with the obvious order, where $C_a <C_b$ means $a <b$). To see this, consider a local function $f \colon  U \to \UI$ where $U$ is an open neighbourhood of $C_a$. We may choose a neighborhood  
$V$ of $C_a$ on $C$ such that $U$ contains any chord joining two points in $V$. Since $f$ is continuous (in particular along $C$) and non-decreasing along these chords (whose endpoints are dense in $V$),  
it has to be non-decreasing along $V$. Thus directed paths on $C$ (equipped with the above order) are weakly directed in $H$ but not directed.

\end{example}

Now we wish to show in which sense our saturation process is maximal among reasonable saturation processes, in the following sense.
 
\begin{defi} A \emph{d-saturation process} is any functor $S \colon
  \DTop \to \DTop$ which commutes with the forgetful functor $\DTop
  \to Top$, equipped with a natural transformation from the identity $
  \eta^S\colon \id \to S$, such that
  \begin{itemize}
  \item $\eta^S$ is mapped to the identity by the forgetful functor to $\Top$;
  \item its component $\eta^S_\UI \colon \UI \to S\UI$ at $\UI$ is the identity;
  \item $S$ satisfies the following ``locality'' condition: for any
    $d$-space $X$ and subspace $Y \subseteq X$, directed paths in $SX$
    with image contained in $Y$ are also directed in $SY$.
  \end{itemize}
\end{defi}
 
\begin{remark}   
Let us comment on the previous condition. First note that directed paths in $SY$ are automatically directed in $SX$ thanks to functoriality. Next let us expain why our condition concerns locality:  
 if $X$ is covered by open subspaces $Y_i$, then $SX$ is determined by the $SY_i$'s. Indeed , by concatenation (and compactness of $\III$),  
a path in $SX$ is directed if and only if each of its restrictions contained in a $Y_i$ is directed in this $SY_i$. (The locality condition is here used in the ``only if'' direction.)
\end{remark}
 
\begin{example} Our functor $L \colon X \mapsto (X, \hat d X)$ is
  obviously a d-saturation process, with $\eta^L$ the unit of $L
  \dashv J$.
\end{example}
 
Now we have an order on d-saturation processes, which says $S \le T$ whenever, for each $X \in \DTop$, the set-theoretic identity of $X$ is directed from $SX$ to $TX$.
Observe in particular that the induced poset contains the (opposite of the) poset of  fully reflective subcategories of $\DTop$.
\begin{theorem} Our functor $L $ is maximal among d-saturation processes.
 
\end{theorem}
 
\begin{proof}
Let us consider a d-saturation process $S$. What we have to prove is that, given a $d$-space $X$, any directed path $c$ in $SX$ is weakly directed in $X$. For this, we take a directed function $f \colon  U \to \UI$
on $X$ and prove that for any closed directed subpath $c'$ of $c$ with image contained in $U$, $f \circ c'$ is non-decreasing. By functoriality of $S$, $f$ is also a morphism from $SU$  
to $\UI = S\UI$, and by locality,  $c'$ is also directed in $SU$, so that $f \circ c'$ is an endomorphism of $\UI$, hence non-decreasing.
\end{proof}
 
 
\bibliographystyle{eptcs} \bibliography{bib}
 
\end{document}